\documentclass[11pt]{amsart}
\usepackage{geometry}                
\usepackage{graphicx}
\usepackage{amssymb}
\usepackage{epstopdf}
\usepackage{amsmath}
\usepackage{color}
\usepackage{hyperref}
\usepackage{pdfsync}
 \usepackage{tikz}
 
\usepackage[all]{xy}
\xyoption{matrix}
\xyoption{arrow}

\usetikzlibrary{arrows,decorations.pathmorphing,backgrounds,positioning,fit,petri}

 \tikzset{help lines/.style={step=#1cm,very thin, color=gray},
help lines/.default=.5} 
\tikzset{thick grid/.style={step=#1cm,thick, color=gray},
thick grid/.default=1} 

\newtheorem{thm}{Theorem}[subsection]
\newtheorem{lem}[thm]{Lemma}
\newtheorem{cor}[thm]{Corollary}
\newtheorem{prop}[thm]{Proposition}

\theoremstyle{definition}
\newtheorem{defn}[thm]{Definition}

\theoremstyle{remark}
\newtheorem{rem}[thm]{Remark}
 
\numberwithin{equation}{section}

\newcommand{\vs}[1]{\vskip .#1 cm} 

\newcommand{\into}{\hookrightarrow}
 \newcommand{\onto}{\twoheadrightarrow}

\newcommand{\mL}{\ensuremath{\cC^m(\Lambda)}}

\newcommand{\brk}[1]{\left<#1\right>}

\DeclareMathOperator{\Hom}{Hom}%
\DeclareMathOperator{\Ext}{Ext}%
\DeclareMathOperator{\End}{End}%
\DeclareMathOperator{\undim}{\underline{dim}}
 
\newcommand{\field}[1]{\mathbb{#1}}
\newcommand{\ZZ}{\ensuremath{{\field{Z}}}}

\newcommand{\QQ}{\ensuremath{{\field{Q}}}}

\newcommand{\commentout}[1]{}

\newcommand{\cA}{\ensuremath{{\mathcal{A}}}}

\newcommand{\cB}{\ensuremath{{\mathcal{B}}}}
\newcommand{\cC}{\ensuremath{{\mathcal{C}}}}
\newcommand{\cD}{\ensuremath{{\mathcal{D}}}}
\newcommand{\cE}{\ensuremath{{\mathcal{E}}}}

\newcommand{\cW}{\ensuremath{{\mathcal{W}}}}

\newcommand\vare{\varepsilon}

\newcommand\noi{\noindent}

\title[m-clusters using exceptional sequences]{Enumerating $m$-clusters using exceptional sequences}
\author{Kiyoshi Igusa}
\address{Department of Mathematics, Brandeis University, Waltham, MA 02454}\email{igusa@brandeis.edu}

 
 
\begin{document}

\begin{abstract} We give a bijection between ordered $m$-clusters and (complete) $m$-exceptional sequences, a concept that we introduce for this purpose. This holds for all hereditary artin algebras. This extends the bijection in the $m=1$ case shown in \cite{1-IT13}.
\end{abstract}

\maketitle

%
%

For any Dynkin quiver $Q$, the number of clusters in the cluster category \cite{3-BMRRT} of any hereditary algebra $\Lambda$ with valued quiver $Q$ is given by the formula
\[
	\prod_{i=1}^n\frac{h+d_i}{d_i}
\]
where $h$ is the Coxeter number of $Q$ and $d_i=2,\cdots,h$ are the degrees of the principal $W$-invariant polynomials where $W$ is the Weyl group of $Q$. The formula was generalized to $m$-clusters by Fomin and Reading \cite{4-FR}:
\[
	p_\Lambda(m)=\prod_{i=1}^n\frac{hm+d_i}{d_i}.
\]
A remarkable paper by Galashin, Lam, Trinh and Williams \cite{GLTW} gives a uniform proof of this formula.

This paper gives another perspective on this formula using exceptional sequences. Namely, we show that the set of ordered $m$-clusters is in 1-1 correspondence with the set of $m$-exceptional sequences (Definition \ref{def: m-exc seq}). This can be viewed as a refinement of the observation that the leading coefficient of $n!p_\Lambda(m)$ is the number of exceptional sequences:
\[
	e_\Lambda=\frac{n!h^n}{\prod d_i}.
\]
This uniform formula for $e_\Lambda$ was first observed by Chapoton \cite{5-Ch} and now has several uniform proofs, e.g. \cite{M}, \cite{CD}, \cite{D}.

This is an old unpublished paper whose main result is being upgraded using a new joint result with Shujian Chen \cite{CI}. A closely related result of Buan, Reiten and Thomas \cite{9-BRT} is discussed in section \ref{ss41}.

The original version of this paper had a discussion of probability distributions of terms in exceptional sequences of type $A_n$. This has been removed and is being published as a separate paper \cite{AnProb}. 

%
%

\section{Basic definitions}\label{sec1}

We recall the definitions and basic properties of exceptional sequences from \cite{6-CB}, \cite{7-Ringel}. Over an hereditary algebra $\Lambda$, a module $M$ is called \emph{exceptional} if $M$ is rigid and Schurian where \emph{rigid} means $\Ext^1_\Lambda(M,M)=0$ and \emph{Schurian} means $\End_\Lambda(M)$ is a division algebra. [Schurian modules are usually called ``bricks."] An \emph{exceptional sequence} of length $\ell$ for $\Lambda$ is a sequence of exceptional modules $(E_1,E_2,\cdots,E_\ell)$ so that
\[
	\Hom_\Lambda(E_j,E_i)=\Ext^1_\Lambda(E_j,E_i)=0\text{ for } i<j\le \ell.
\]
The exceptional sequence is called \emph{complete} if it is of maximal length and this is well known to occur when $\ell=n$, the number of nonisomorphic simple $\Lambda$-modules. We consider exceptional sequences up to isomorphism where two sequences $(E_i)$, $(E_i')$ are isomorphic if they have the same length and $E_i\cong E_i'$ for all $i$.

\subsection{Relative projective terms}\label{ss11}
In each exceptional sequence we will define certain terms to be ``relatively projective.''

For any $\Lambda$-module $M$ the (right) \emph{perpendicular category} $M^\perp$ is defined to be the full subcategory of $mod\text-\Lambda$ of all objects $X$ so that $\Hom_\Lambda(M,X)=0=\Ext^1_\Lambda(M,X)$. The left perpendicular category $^\perp M$ is defined analogously. From the six-term exact sequence for $\Hom_\Lambda(M,-)$ it follows that $M^\perp$ and $^\perp M$ are closed under extensions, kernels and cokernels of morphisms between objects. Any full subcategory of $mod\text-\Lambda$ having these properties is called a \emph{wide subcategory}. It is well-known that the perpendicular categories $M^\perp$, $^\perp M$ are finitely generated, i.e., contain a generator. Conversely, if a wide subcategory $\cW$ of $mod\text-\Lambda$ has a generator $X$ then $^\perp X$ has a generator, say $Y$, and $\cW=Y^\perp$. Therefore, perpendicular categories are the same as finitely generated wide subcategories of $mod\text-\Lambda$.

A term $E_j$ in an exceptional sequence $(E_1,\cdots,E_\ell)$ will be called \emph{relatively projective} if it is a projective object of the perpendicular category of the later terms, $(E_{j+1}\oplus\cdots\oplus E_\ell)^\perp$. In a complete exceptional sequence, the first term $E_1$ is always relatively projective since the perpendicular category $(E_2\oplus\cdots\oplus E_n)^\perp$ is semi-simple. If $\Lambda$ has finite representation type then we define $f_\Lambda(x)$ to be the degree $n$ integer polynomial given by
\[
	f_\Lambda(x):=\sum_{k=1}^n e_kx^k
\]
where $e_k$ is the number of complete exceptional sequences in $mod\text-\Lambda$ with $k$ relatively projective terms. This polynomial is suitable for computing the number of $m$-clusters in the $m$-cluster category of $\Lambda$ defined below. This polynomial also gives the number of complete exceptional sequences which, following \cite{2-Ringel13}, we denote by $e_\Lambda= f_\Lambda(1)=\sum e_k$.

\subsection{$m$-exceptional sequences}\label{ss12}

Recall that, since $\Lambda$ is hereditary, the indecomposable objects of the bounded derived category of $mod\text-\Lambda$ have the form $M[k]$ where $M$ is an indecomposable $\Lambda$-module and $k\in\ZZ$. We say that $M[k]$ is \emph{exceptional} if $M$ is an exceptional $\Lambda$-module. The following definition extends an idea from \cite{1-IT13}.

\begin{defn}\label{def: m-exc seq}
For any $m\ge0$ we define an \emph{$m$-exceptional sequence} of length $\ell$ for $\Lambda$ to be a sequence of exceptional objects $(E_1,\cdots,E_\ell)$ in $\cD^b(mod\text-\Lambda)$ so that
\begin{enumerate}
\item for all $j$, $E_j=M_j[d_j]$ where $0\le d_j\le m$, 
\item $(M_1,\cdots,M_\ell)$ is an exceptional sequence in $mod\text-\Lambda$ and 
\item when $d_j=m$, the module $M_j$ is relatively projective in the perpendicular category $(M_{j+1}\oplus \cdots\oplus M_\ell)^\perp$. 
\end{enumerate}
An $m$-exceptional sequence will be called \emph{complete} if it has the maximum length $\ell=n$. For $m=1$, an $m$-exceptional sequence is called a \emph{signed exceptional sequence} \cite{1-IT13}.
\end{defn}

It is clear from this definition that, to every exceptional sequence $(E_1,\cdots,E_\ell)$ with $k$ relatively projective terms, there correspond $(m+1)^k m^{\ell-k}$ $m$-exceptional sequences since there are $m+1$ possibilities for $d_j$ when $M_j$ is relatively projective and only $m$ possibilities otherwise. Therefore, for $\Lambda$ of finite type, the number of complete $m$-exceptional sequences is given by
\[
	g_\Lambda(m):=\sum_{k=1}^n e_k(m+1)^km^{n-k}=	m^n f_\Lambda\left(
	\frac{m+1}m
	\right).
\]

We also consider the \emph{$m$-cluster category of $\Lambda$} \cite{8-mT}. This is defined to be the orbit category of the bounded derived category of $mod\text-\Lambda$ modulo the functor $\tau^{-1}[m]$:
\[
	\cC^m(\Lambda)=\cD^b(mod\text-\Lambda)/\tau^{-1}[m]
\]
We identify objects of $\cC^m(\Lambda)$ with their representative in $\cD^b(mod\text-\Lambda)$ in the fundamental domain of $\tau^{-1}[m]$ which we take to be
\begin{equation}\label{eq: m-cluster cat, FD}
	mod\text-\Lambda\cup mod\text-\Lambda[1]\cup \cdots\cup mod\text-\Lambda[m-1]\cup  \Lambda[m]
\end{equation}
Thus the indecomposable objects of $\cC^m(\Lambda)$ are represented by $M[j]$ for some indecomposable $\Lambda$-module $M$ where $0\le j\le m$ with the additional condition that $M$ is projective when $j=m$. Two object $X[j], Y[k]$ in $\cC^m(\Lambda)$ are \emph{compatible} if $\Ext_{\cD^b}^i(X[j],Y[k])=\Ext_{\cD^b}^i(Y[k],X[j])=0$ for all $i\ge1$ and they are not isomorphic. Equivalently, either
\begin{enumerate}
\item $j<k$ and $\Hom_\Lambda(Y,X)=0=\Ext^1_\Lambda(Y,X)$,
\item $j=k$ and $\Ext^1_\Lambda(X,Y)=0=\Ext^1_\Lambda(Y,X)$ and $X\not\cong Y$ or
\item $j>k$ and $\Hom_\Lambda(X,Y)=0=\Ext^1_\Lambda(X,Y)$,
\end{enumerate}
A maximal pairwise compatible set of exceptional objects of $\cC^m(\Lambda)$ has $n$ objects and is called an \emph{$m$-cluster} for $\Lambda$.

The first part of the main theorem of this paper, proved in Section \ref{sec: main theorem} below, is the following which generalizes the $m=2$ case proved in the preliminary version of \cite{1-IT13}. This new proof supersedes the one in \cite{1-IT13}.

\begin{thm}\label{main thm, part 1}
Let $\Lambda$ be a finite dimensional hereditary algebra over any field and $k,m\ge0$. Then there is a 1-1 correspondence between (isomorphism classes of) ordered $k$-tuples of pairwise compatible objects in the $m$-cluster category of $\Lambda$ and (isomorphism classes of) $m$-exceptional sequences of length $k$ for $\Lambda$.
\end{thm}

For example, when $m=0$, there is only one $0$-cluster of size $n$, namely the set of projective modules. So, there are $n!$ complete 0-exceptional sequences. These are the same as complete exceptional sequences in which every object is relatively projective. It follows that the number of maximal compatible subsets of $\cC^m(\Lambda)$, the \emph{$m$-clusters}, is equal to
\[
	\frac1{n!}g_\Lambda(m)=\frac{m^n}{n!} f_\Lambda\left(
	\frac{m+1}m
	\right)
\]
The case of greatest interest is when $m=1$. The number of cluster for the algebra $\Lambda$ is
\[
	\frac1{n!} f_\Lambda\left(
	2
	\right)
\]
where $f_\Lambda(2)$ is the number of signed exceptional sequences.

Some of the facts that we can see immediately from this formula are:

\begin{enumerate}
\item $g_\Lambda(-1)=f_\Lambda(0)=0$ (since $E_1$ is always relatively projective)
\item $g_\Lambda(0)=e_n=n!$ (the number of ordered $m$-clusters for $m=0$).
\end{enumerate}

\begin{prop}\label{prop: where are roots of g(x)}
All real roots of the polynomial $g_\Lambda(x)$ are between $0$ and $-1$ including $-1$ and excluding $0$. 
\end{prop}

\begin{proof}
If $m\ge 0$ then $g_\Lambda(m)>0$ since all terms are nonnegative: $e_k(m+1)^km^{n-k}\ge0$ and the $k=n$ term is $n!(m+1)^n\ge n!>0$. If $m<-1$ then $(-1)^ng_\Lambda(m)>0$ since all terms are nonnegative and the $k=n$ term is positive.
\end{proof}

The second part of the main theorem, also proved in Section \ref{sec: main theorem} is the following.

\begin{prop}\label{main thm, part 2}
The 1-1 correspondences of Theorem \ref{main thm, part 1} can be chosen to be compatible with deletion of the first term in the sense that, if $(T_1,\cdots,T_k)$ corresponds to $(E_1,\cdots,E_k)$, then $(T_2,\cdots,T_k)$ corresponds to $(E_2,\cdots,E_k)$.
\end{prop}

\subsection{Recursive formula}\label{ss13}

We will show that the recursive formula for the number of exceptional sequences given in \cite{2-Ringel13} can be modified to give a recursive formula for $f_\Lambda(x)$.

We review briefly the argument from \cite{2-Ringel13} which counts the number of exceptional sequences
\[
	e_\Lambda=\sum e_k=f_\Lambda(1)
\]

Let $\Lambda$ be a hereditary algebra of finite type with connected valued quiver $Q$. The number of vertices of $Q$, usually denoted $n$, is the \emph{rank} of $\Lambda$. For each vertex $i$ of $Q$, let $Q(i)$ be the valued quiver obtained from $Q$ by deleting the vertex $i$. Then $Q(i)$ has $1,2$ or $3$ components. To count the number of exceptional sequences for $Q(i)$ we need the following lemma.

\begin{lem}\cite{2-Ringel13}\label{lem: product formula of Ringel}
Let $A,B$ be hereditary algebras of finite type with rank $n_A,n_B$. Then there is a 1-1 correspondence between complete exceptional sequences for $A\times B$ and pairs of complete exceptional sequences for $A,B$ together with a shuffling of the two sequences. Thus
\[
	e_{A\times B}=\binom{n_A+n_B}{n_A} e_Ae_B.
\]
\end{lem}

Since a term in an exceptional sequence for $A\times B$ is relatively projective if and only if it is relatively projective in $A$ or $B$ whichever it comes from, we get the following.

\begin{equation}\label{formula for A x B}
f_{A\times B}(x)=\binom{n_A+n_B}{n_A} f_A(x)f_B(x)
\end{equation}

The recursive formula for $e_\Lambda$ is:

\begin{prop}\cite{2-Ringel13}\label{prop: recursive formula for eL}
For $\Lambda$ of rank $n$ and Coxeter number $h$ we have:
\[
	e_\Lambda=\sum_{i=1}^n \frac h2 e_{\Lambda(i)}
\]
\end{prop}

\begin{proof} (Summarized from \cite{2-Ringel13}) If $P_i$ and $I_j$ are in the same orbit of the Auslander-Reiten translation $\tau$ then the sum of the lengths of the $\tau$ orbits of $P_i$ and $P_j$ is $h$. Furthermore, for any indecomposable module $M$ in the union of these $\tau$ orbits, $M^\perp$, $P_i^\perp$ and $P_j^\perp$ are all isomorphic to $\Lambda(i)$. Since this counts each indecomposable module twice and since every complete exceptional sequence for $\Lambda$ is given by an indecomposable module $M$ preceded by a complete exceptional sequence for $M^\perp$, we get:
\[
	2e_\Lambda=2\sum_M e_{M^\perp}=h\sum_i e_{\Lambda(i)}.
\]
The formula in the Proposition follows.
\end{proof}

In the above proof, in the union of the $\tau$ orbits of $P_i$ and $P_j$, the proportion of objects which are projective is exactly $\frac2h$. This gives the following.

\begin{cor}\label{cor: recursive formula for f(x)}
Let $\Lambda,n,h$ be as above. Then
\[
	f_\Lambda(x)=\left(
	x+\frac h2-1
	\right)\sum_{i=1}^n f_{\Lambda(i)}(x)
\]
or, equivalently,
\[
	{g_\Lambda(m)=\frac{hm+2}2 \sum_{i=1}^n g_{\Lambda(i)}(m)}.
\]
\end{cor}

%
%

\section{Proof of main theorem}\label{sec: main theorem}

This section gives the statement and proof of the main theorem.

\subsection{Statement of main theorem}\label{ss31}

\begin{thm}\label{main thm}
There is a bijection between the set of isomorphism classes of $m$-exceptional sequences of length $k$ and the set of isomorphism classes of ordered $k$-tuples of pairwise compatible objects in the $m$-cluster category for any finite dimensional hereditary algebra $\Lambda$. Furthermore, this bijection is compatible with deletion of the first term in the sense that, if $(E_1,\cdots,E_k)$ corresponds to $(T_1,\cdots,T_k)$, then $(E_2,\cdots,E_k)$ corresponds to $(T_2,\cdots,T_k)$.
\end{thm}

We will set up notation to give an explicit formula for this bijection. We choose one object from every isomorphism class of exceptional $\Lambda$-modules. Let $\cE(\Lambda)$ denote the set of these objects. For set theoretic purposes we can consider this to be a subset of $\ZZ^n$ since every exceptional module $M$ is determined up to isomorphism by its dimension vector $\undim M\in\ZZ^n$.

For any finitely generated wide subcategory $\cW$ of $mod\text-\Lambda$ (defined in section \ref{ss11}), let $\cE(\cW)=\cW\cap \cE(\Lambda)$. Denote by $\cC^m(\cW)$ the $m$-cluster category of $\cW$. Objects of $\cC^m(\cW)$ will be represented by elements of the fundamental domain of the functor $\tau^{-1}[m]$ on the bounded derived category of $\cW$, as in \eqref{eq: m-cluster cat, FD}. Thus the exceptional objects of $\cC^m(\cW)$ are (up to isomorphism) $X[j]$ where $X\in\cE(\cW)$ and $0\le j\le m$ with the additional condition that $X$ is a relative projective object of $\cW$ when $j=m$. Let $\cE^m(\cW)$ denote this set of objects. We call $j$ the \emph{level} of the object $X[j]$. Let $\cE^m_j(\cW)$ denote the set of objects of $\cE^m(\cW)$ of level $j$. In particular, the \emph{rank} $r\le n$ of $\cW$ is the number of elements of $\cE_m^m(\cW)$. We use the abbreviation $\cE^m(\Lambda)=\cE^m(mod\text-\Lambda)$. We consider the special case $\cW=T^\perp$ where $T\in\cE(\Lambda)$. This is a wide subcategory of rank $n-1$. We note that $T$ is uniquely determined by $T^\perp$ since $T$ is the unique exceptional object of $^\perp(T^\perp)$.

We use the expression \emph{exceptional pair} for an exceptional sequence of length $2$ in $mod\text-\Lambda$. We say that two are \emph{equivalent}: $(X,Y)\sim (Z,W)$ if $(X\oplus Y)^\perp=(Z\oplus W)^\perp$, i.e., if they \emph{span} the same rank 2 wide subcategory $\,^\perp\left((X\oplus Y)^\perp\right)$ of $mod\text-\Lambda$. 


\subsection{The key lemma}\label{ss32}
We state the key lemma, which uses the following notation, and show how it proves the main theorem. 

For any $T[k]\in\cE^m(\Lambda)$, let $\cE^{T[k]}\subseteq \cE^m(\Lambda)$ denote the set of all $X[j]\in \cE^m(\Lambda)$ which are compatible with $T[k]$. Recall that $A,B$ are compatible if $A\not\cong B$ and $\Hom(A,B[s])=0=\Hom(B,A[s])$ for all $s>0$. For $A=X[j],B=T[k]$ this is equivalent to one of the following.\begin{enumerate}
\item $j<k$ and $(X,T)$ is an exceptional pair,
\item $j>k$ and $(T,X)$ is an exceptional pair or
\item $j=k$ and $T,X$ are ext-orthogonal and nonisomorphic in $mod\text-\Lambda$.
\end{enumerate}

\begin{lem}\label{key lemma}
For any $T[k]\in\cE^m(\Lambda)$ there is a bijection
\[
	\sigma_{T[k]}:\cE^m(T^\perp)\to\cE^{T[k]}.
\]
Furthermore, $A,B$ are compatible in $\cE^m(T^\perp)$ if and only if $\sigma_{T[k]}A,\sigma_{T[k]}B$ are compatible in $\cE^{T[k]}$.
\end{lem}

\begin{proof}[Proof of Theorem \ref{main thm} given Lemma \ref{key lemma}]
We show that the bijections $\sigma_{T[k]}$ give bijections 
\[
	\theta_p:\{\text{$p$-tuples of compatible objects in $\cE^m(\Lambda)$}\}\to \{m\text{-exceptional sequences of length $p$ for $\Lambda$}\}
\]
compatible with deletion of the first term. The bijections $\theta_p$ are defined recursively as follows. For $p=1$ let $\theta_1$ be the identity mapping. Next, take $p\ge2$.

Let $(T_1,\cdots,T_p)$ be an ordered $p$-tuple of compatible objects in $\cE^m(\Lambda)$ and let $T_p=T[k]$ be the last object. Then, $T_i$ for $i<p$ are compatible elements of $\cE^{T[k]}$. So, by the lemma, $\sigma_{T[k]}^{-1}(T_i)$ are compatible objects of $\cE^m(T^\perp)$. So, $\sigma_{T[k]}^{-1}(T_1,\cdots,T_{p-1})=(X_1,\cdots,X_{p-1})$ is an ordered $(p-1)$-tuple of compatible objects in the wide subcategory $T^\perp$ of $mod\text-\Lambda$. Since $T^\perp$ is isomorphic to $mod\text-\Lambda'$ for some hereditary algebra $\Lambda'$, by induction on $p$, we have the bijection
\[
	\theta_{p-1}:\{\text{$(p-1)$-tuples of compatible objects in $\cE^m(T^\perp)$}\}\to \{m\text{-exc. seq. of length $p-1$ for $T^\perp$}\}
\]
which is given by $\theta_{p-2}$ when first terms are deleted. Let
\[
	\theta_{p-1}\left(
	X_1,\cdots,X_{p-1}
	\right)=(E_1,\cdots,E_{p-1})
\]
be the corresponding $m$-exceptional sequence in $T^\perp$. Then we define $\theta_p$ by
\[
	\theta_p(T_1,\cdots,T_p)=(E_1,\cdots,E_{p-1},T[k]).
\]
Since $\theta_{p-2}(T_2,\cdots,T_{p-1})=(X_2,\cdots,X_{p-1})$, it follows that $\theta_{p-1}(T_2,\cdots,T_p)=(E_2,\cdots,E_p)$ as required.

Conversely, let $(E_1,\cdots,E_p)$ be an $m$-exceptional sequence for $\Lambda$. Let $E_p=T[k]$. Then $(E_1,\cdots,E_{p-1})$ is an $m$-exceptional sequence in $T^\perp$ and
\[
	\theta_{p-1}^{-1}(E_1,\cdots,E_{p-1})=(X_1,\cdots,X_{p-1})
\]
is a $(p-1)$-tuple of compatible objects in $T^\perp$. 
Applying $\sigma_{T[k]}$ gives $\sigma_{T[k]}(X_1,\cdots,X_{p-1})=(T_1,\cdots,T_{p-1})$ a $(p-1)$-tuple of compatible objects in $\mL$. So, $\theta_p^{-1}(E_1,\cdots,E_p)=(T_1,\cdots,T_p)$. By construction the maps $\theta_p$, $\theta_p^{-1}$ are inverse to each other. Since the maps $\theta_p$ are compatible with deletion of the first term, so are the maps $\theta_p^{-1}$. Thus the key lemma implies the main theorem.
\end{proof}


\subsection{Outline of proof of key lemma}\label{ss33}

We will construct the bijection
\[
	\sigma_{T[k]}:\cE^m(T^\perp)\to \cE^{T[k]}
\]
by decomposing each set into a disjoint union of subsets and constructing bijections between corresponding subsets as shown schematically as follows.
\begin{equation}\label{bijection chart}
\begin{array}{ccccccccccccccc}
\cE^m(T^\perp)&=&\cE_0 & \cup \cdots  \cup & \cE_{k-1} &\cup &\cB_k&\cup &\cA_{k+1}&\cup&\cB_{k+1}&\cup\cdots\cup&\cA_m&\cup&\cB_m\\
&&\downarrow &&  \downarrow &&\downarrow&&\downarrow&&\downarrow&&\downarrow &&\downarrow \\
\cE^{T[k]}&=&\cE_0' & \cup \cdots \cup & \cE_{k-1}'&\cup &\cB_k'&\cup &\cA_k'&\cup&\cB_{k+1}'&\cup\cdots\cup&\cA_{m-1}'&\cup&\cB_m'
\end{array}
\end{equation}
Here $\cE_j=\cE^m_j(T^\perp)$ and $\cE_j'=\cE_j^{T[k]}$ denote the subsets of $\cE^m(T^\perp)$ and $\cE^{T[k]}$, resp, of objects of level $j$. It follows from the definitions that, for $j<k$, $\cE_j=\cE_j'$. So, we take the identity mapping
\[
	\sigma_{T[k]}|_{\cE_j}=id:\cE_j\xrightarrow= \cE_j'\,, \quad 0\le j<k.
\]

With the exception of the bijection $\cB_k\to \cB'_k$, all other bijections in Chart \ref{bijection chart} are given by mutation of exceptional sequences $(X,T)\leftrightarrow(T,Y)$. This gives a bijection $\gamma:\cE(T^\perp)\cong \cE(^\perp T)$ where the formula for $Y=\gamma(X)$ is given by examining four cases. We pick out one case of this bijection. Let $\cA\subseteq \cE(T^\perp)$ and $\cA'\subseteq \cE(^\perp T)$ be given by
\[
	\cA=\{X\in\cE(T^\perp)\,:\, \exists \text{ mono }X\into T^s\}
\]
\[
	\cA'=\{Y\in\cE(^\perp T) \,:\, \exists \text{ epi }T^s\onto Y\}.
\]
The bijection $\gamma:\cE(T^\perp)\cong \cE(^\perp T)$ sends $\cA$ onto $\cA'$ and gives a bijection $\alpha:\cA\cong \cA'$. Corresponding objects $X\in\cA\leftrightarrow Y\in \cA'$ are related by the canonical short exact sequence
\[
	0\to X\xrightarrow f T^s\xrightarrow g Y\to 0
\]
where $f,g$ are minimal left/right $T$-approximations of $X,Y$ respectively. Let $\cB=\cE(T^\perp)-\cA$ and $\cB'=\cE(^\perp T)-\cA'$ be the union of the other cases. Thus, $\gamma$ induces a bijection $\beta:\cB\cong \cB'$.

For $j>k$ we define $\sigma_{T[k]}(X[j])$ by
\begin{equation}\label{formula for s}
	\sigma_{T[k]}(X[j])=\begin{cases} \alpha(X)[j-1] & \text{if } X\in \cA\\
   \beta(X)[j] & \text{otherwise}
    \end{cases}
\end{equation}
We will show that this gives a bijection
\[
	\sigma_{T[k]}|_{\cA_j}=\alpha_j:\cA_j\to \cA_{j-1}'
\]
between $\cA_j=\{X[j]\,:\, X\in\cA\}$ and $\cA_{j-1}'=\{Y[j-1]\,:\, Y\in\cA'\}$ for $k<j\le m$. We will verify that $\cA_j \subseteq \cE^m_j(T^\perp)$ for all $k< j\le m$ and $\cA_j'\subseteq \cE_j^{T[k]}$ for all $k\le j<m$.

We will show that, for $k<j\le m$, Formula \ref{formula for s} also gives a bijection
\[
	\sigma_{T[k]}|_{\cB_j}=\beta_j:\cB_j\to\cB_j'
\]
where $\cB_j=\cE_j-\cA_j$ and $\cB_j'=\cE_j'-\cA_j'$.

For $j=k\le m$ let $\cB_k=\cE_k$ and $\cB_k'=\cE_k'-\cA_k'$. The bijection 
\[
	\sigma_{T[k]}|_{\cB_k}=\beta_k:\cB_k\to \cB'_k
\]
is constructed by a special argument which can be summarized by saying that $\beta_k(X[k])$ is either $X[k]$ or $\gamma(X)[k]$, whichever lies in the set $\cB_k'$. A similar description summarizes the bijection $\sigma_{T[k]}$ on all objects of $\cE^m(T^\perp)$ (See Proposition \ref{prop summarizing bijection sigma-T[k]}).

This completes the description of each object and each morphism in Chart \ref{bijection chart}. The final step is to show that $\sigma_{T[k]}$ preserves compatibility. I.e., $X[j],X'[j']$ are compatibility in $\cE^m(T^\perp)$ if and only if $\sigma_{T[k]}(X[j])$, $\sigma_{T[k]}(X'[j'])$ are compatibility in $\cE^m{\Lambda}$. This will follow from the fact that $X,X',T$, appropriately order, form an exceptional sequence.


\subsection{The bijection $\alpha_j:\cA_j\to\cA_{j-1}'$}\label{ss34}

Recall from the previous section that $\cA$ is the set of all $X\in\cE(T^\perp)$ so that there is a monomorphism $X\into T^s$. We need the following observation.

\begin{rem}\label{X is rel projective} Any monomorphism $X\into T^s$ induces an epimorphism $\Ext^1_\Lambda(T^s,Z)\onto \Ext^1_\Lambda(X,Z)$ for any $Z$ since $\Lambda$ is hereditary. In particular, $\Ext^1_\Lambda(X,Z)=0$ for $Z\in T^\perp$. For $X\in T^\perp$ this implies that $X$ is a relatively projective object of $T^\perp$. 
\end{rem}

The standard bijection $\gamma:\cE(T^\perp)\cong \cE(^\perp T)$ sends any $X\in\cA$ to $\gamma(X)=Y$, the cokernel of the minimal $T$-approximation of $X$:
\[
	0\to X\xrightarrow fT^s\to Y\to 0
\]
We call this bijection $\alpha:\cA\to\cA'=\gamma(\cA)$. Recall (from the previous section) that, for $k<j\le m$, $\cA_j$ denotes the set of all $X[j]$ where $X\in\cA$ and, for $k\le j<m$, $\cA_j'$ denotes the set of all $Y[j]$ where $Y\in\cA'$. 

\begin{rem}\label{Y is rel injective}
By Remark \ref{X is rel projective} any $X\in\cA$ is relatively projective in $T^\perp$ and, by the dual argument, any $Y\in\cA'$ is relatively injective in $^\perp T$.
\end{rem}

\begin{lem}\label{lem: where is Aj}
$\cA_j \subseteq \cE^m_j(T^\perp)$ for all $k< j\le m$ and $\cA_j'\subseteq\cE_j^{T[k]}$ for all $k\le j<m$. 
\end{lem}

\begin{proof}
Since $(X,T)$ is exceptional, $X[j]\in \cE^m_j(T^\perp)$ for $j<m$. For $j=m$ we have, by Remark \ref{X is rel projective}, that $X[m]\in \cE^m_m(T^\perp)$. Thus $\cA_j \subseteq \cE^m_j(T^\perp)$ for all $k< j\le m$.

For $k<j<m$, the condition of $(T,Y)$ being exceptional is equivalent to $Y[j]\in \cE_j^{T[k]}=\cE_j(^\perp T)$. For $j=k$, the additional condition that $\Hom_\Lambda(T,Y)\neq0$ implies that $\Ext_\Lambda^1(T,Y)=0$. Therefore $Y,T$ are ext-orthogonal modules or, equivalently, $Y[k]\in \cE_k^{T[k]}$. So, for all $k\le j<m$, $\cA_j'\subseteq\cE_j^{T[k]}$. 
\end{proof}

The bijection $\alpha_j:\cA_j\to \cA_{j-1}'$ is given by sending $X[j]$ to $Y[j-1]$ where $Y=\alpha(X)$.


\subsection{The bijection $\beta_j:\cB_j\to \cB_{j}'$}\label{ss35}

We recall the notation from the outline: $\cB=\cE(T^\perp)-\cA$ and $\cB'=\cE(^\perp T)-\cA'$. The bijection $\gamma:\cE(T^\perp)\cong \cE(^\perp T)$ will be denoted with the letters $X,Y$. Thus $Y=\gamma(X)$. There are four cases, but one of them is $\cA\cong \cA'$.

\begin{rem}\label{chart for bijection b}
The bijection $\beta:\cB\cong \cB'$ falls into three cases.
\[
\begin{array}{cccc}
&X\in \cB &\leftrightarrow& Y\in\cB'\\
\hline\\
(a) & \text{ $X,T$ are hom-ext orthogonal } & X=Y & \text{ $T,Y$ are hom-ext orthogonal }\\
(b) & \Ext_\Lambda^1(X,T)\neq 0 & T^s\into Y\onto X & \exists g:T^s\into Y\\
(c) & \exists f:X\onto T^s & Y\into X\onto T^s & \Ext^1_\Lambda(T,Y)\neq0
\end{array}
\] 
\end{rem}

\begin{lem}\label{extension to k=m}
The bijection $\beta:\cB\to \cB'$, sending $X$ to $Y$, has the property that $Y$ is projective in $mod\text-\Lambda$ if and only if $X$ is a relatively projective object in $T^\perp$.
\end{lem}

\begin{proof}(from \cite{1-IT13})
In all three cases, we have an isomorphism
\[
	\Ext_\Lambda^1(X,Z)\cong \Ext_\Lambda^1(Y,Z)
\]
for all $Z\in T^\perp$. This is trivial when $X=Y$ and follows from the six term $\Hom\text-\Ext$ series in the other cases. If $Y$ is projective these are zero and thus $X$ is relatively projective in $T^\perp$.

Conversely, suppose that $X$ is relatively projective in $T^\perp$ and $Y$ is not projective. Then we will obtain a contradiction. Let $Z\in mod\text-\Lambda$ be minimal so that $\Ext^1_\Lambda(Y,Z)\neq0$. Since $\Ext_\Lambda^1(X,-)\cong \Ext_\Lambda^1(Y,-)$ on $T^\perp$ we have $Z\notin T^\perp$. Therefore, either
\begin{enumerate}
\item $\Hom_\Lambda(T,Z)\neq0$ or
\item $\Ext^1_\Lambda(T,Z)\neq0$ (and $\Hom_\Lambda(T,Z)=0$).
\end{enumerate}
In Case (1) there is a right exact sequence $T\to Z\to W\to 0$ where $W$ is smaller than $Z$. Since $Y\in \,^\perp T$, we have $\Ext_\Lambda^1(Y,Z)\cong \Ext_\Lambda^1(Y,W)\neq0$ contradicting the minimality of $Z$.

In Case (2), we consider the universal extension $Z\into E\onto T^s$ given by choosing a basis for $\Ext_\Lambda^1(T,Z)$ over the division algebra $\End_\Lambda(T)$. This gives the six term exact sequence
\[
	\Hom_\Lambda(T,Z)\to \Hom_\Lambda(T,E)\to \Hom_\Lambda(T,T^s)\xrightarrow\approx \Ext^1_\Lambda(T,Z)\to \Ext^1_\Lambda(T,E)\to \Ext^1_\Lambda(T,T^s)
\]
where the middle map is an isomorphism by construction and the two end terms are zero: $\Hom_\Lambda(T,Z)=0$ by assumption and $\Ext^1_\Lambda(T,T^s)=0$ since $T$ is exceptional. Therefore, the other two terms are zero and $E\in T^\perp$. But this implies $\Ext^1_\Lambda(Y,E)\cong \Ext^1_\Lambda(X,E)=0$. We have another exact sequence
\[
	\Hom_\Lambda(Y,T^s)\to \Ext^1_\Lambda(Y,Z)\to \Ext^1_\Lambda(Y,E)
\]
which gives $\Ext^1_\Lambda(Y,Z)=0$ since $Y\in\,^\perp T$. This contradiction completes the proof.
\end{proof}

We recall our notation. For all $k<j<m$, $\cB_j$ denotes the set of all $X[j]$ where $X\in\cB$ and $\cB_j'$ denotes the set of all $Y[j]$ where $Y\in\cB'$. Then $\cB_j\subseteq \cE_j$ and $\cB_j'\subseteq \cE_j'$ and we have a bijection $\beta_j:\cB_j\cong \cB_j'$ given by $\beta_j(X[j])=(\beta(X))[j]$. 

For $j=m>k$, $\cB_m=\cE_m-\cA_m$. This is the set of all $X[m]$ where $X\in\cB$ and $X$ is also relatively projective in $T^\perp$. We also have our notation: $\cB_m'=\cE_m'$. This is the set of all $Y[m]$ where $Y\in \cB'$ and $Y$ is also a projective module. By Lemma \ref{extension to k=m}, the bijection $\beta:\cB\cong \cB'$ gives a bijection $\beta_m:\cB_m\cong \cB_m'$ by $\beta_m(X[m])=(\beta(X))[m]$.

For $j=k\le m$, $\cB_k=\cE_k$ including objects $X[k]$ where $X\in\cA$. $\cB_k'=\cE'_k-\cA'_k$ is the set of all objects $Z[k]$ where $Z,T$ are ext-orthogonal and there does not exist an epimorphism $T^s\onto Z$ since such objects $Z$ lie in $\cA'$. Note that $\cA_m'$ is empty since there is no epimorphism $T^s\onto Z$ when $Z$ is projective.

\begin{lem}\label{lem: bijection bk}
There is a bijection $\beta_k:\cB_k\cong \cB_k'$ given by
\[
\beta_k(X[k])=\begin{cases} X[k] & \text{if }\Ext^1_\Lambda(X,T)=0\\
   (\beta(X))[k] & \text{otherwise}
    \end{cases}
\]
with inverse $\beta_k^{-1}$ given by
\[
\beta_k^{-1}(Y[k])=\begin{cases} Y[k] & \text{if }\Hom_\Lambda(T,Y)=0\\
   (\beta^{-1}(Y))[k] & \text{otherwise}
    \end{cases}
\]
\end{lem}

\begin{proof}
When $\Hom_\Lambda(T,Y)\neq0$ we know that $Y\in \,^\perp T$. So, the right $T$ approximation of $Y$ is either a monomorphism $g:T^s\into Y$ or an epimorphism $T^s\onto Y$ and the second case is excluded since $Y\in\cA'$ in that case. Then $\beta^{-1}(Y)=X$ is the cokernel of $g$ and $\beta(X)=Y$. The following chart lists all possibilities.
\[
\begin{array}{cccc}
&X[k]\in \cB_k &\leftrightarrow& Y[k]\in\cB_k'\\
\hline\\
(a) & \text{ $X,T$ are hom-ext orthogonal } & X=Y & \text{ $T,Y$ are hom orthogonal }\\
(b) & \Ext_\Lambda^1(X,T)\neq 0 & T^s\into Y\onto X & \exists g:T^s\into Y\\
(c) & \exists f:X\onto T^s & X=Y& \exists f:Y\onto T^s\\
(d) & \exists h:X\into T^s & X=Y & \exists h:Y\into T^s
\end{array}
\] 

The special case $k=m$ needs further checking. In this case $T$ is projective and Case (c) will not occur. In Case (d), $X$ is projective since it is a submodule of a projective module. So, we can take $Y=X$. Cases (a) and (b) are the same as in Remark \ref{chart for bijection b}. So, Lemma \ref{extension to k=m} tells us that $X$ is relatively projective in $T^\perp$ if and only if $Y$ is a projective module. Therefore, the above formula gives a bijection $\cB_m\to \cB_m'$ when $k=m$.
\end{proof}

Combining all of these case, we obtain the desired bijection as summarized by the following.

\begin{prop}\label{prop summarizing bijection sigma-T[k]}
{
The bijection $
	\sigma=\sigma_{T[k]}:\cE^m(T^\perp)\cong \cE^{T[k]}
$ is given by
\begin{enumerate}
\item $\sigma(X[i])=X[i]$ if $X[i]\in \cE^{T[k]}$.
\item If $X[i]\notin \cE^{T[k]}$ then $\sigma(X[i])=Y[j]$ where $Y=\gamma(X)\in \cE(^\perp T)$ and $j=i$ or $i-1$ so that $(-1)^i\undim X$ and $(-1)^j\undim Y$ are congruent modulo $\undim T$.
\end{enumerate}
} 
\end{prop}


\subsection{Compatibility}\label{ss36}

We will prove that $\sigma=\sigma_{T[k]}:\cE^m(T^\perp)\cong \cE^{T[k]}$ preserves compatibility:

\begin{prop}\label{sigma preserves compatibility}
$X[i],X'[i']$ are compatible objects of $\cE^m(T^\perp)$ if and only if $Y[j]=\sigma(X[i])$ and $Y'[j']=\sigma(X'[i'])$ are compatible in $\cE^{T[k]}$.
\end{prop}

We need the following lemmas.

{
\begin{lem}\label{three exceptional pairs}
Let $X,X'\in\cE(T^\perp)$ and consider the pairs:
\[
	(X,X'),\quad (X,\gamma(X')),\quad (\gamma(X),\gamma(X'))
\]
If any of these is an exceptional pair, the other two are exceptional pairs.
\end{lem}

\begin{proof}
A pair is exceptional if the corresponding triple in the following sequence is exceptional:
\[
	(X,X',T),\quad (X,T,\gamma(X')),\quad (T,\gamma(X),\gamma(X')).
\]
But any one of these gives the other two by mutation of exceptional sequences.
\end{proof}

\begin{lem}\label{linear relation}
Let $(X,X')$ be an exceptional pair in $\cE(T^\perp)=\cA\cup \cB$ and let $Y=\gamma(X),Y'=\gamma(X')\in \cE(^\perp T)=\cA'\cup \cB'$. Let $X''\in\cE(T^\perp)$, $Y''\in\cE(^\perp T)$ be given by mutation of exceptional sequences: 
\[
	(X,X',T)\sim (X',X'',T)\sim (X',T,Y'')\sim (T,Y',Y'')\sim (T,Y,Y')
\]
Then there are unique rational numbers $a,b$ and signs $\vare,\vare'=\pm1$ so that
\[
	\undim X+a\undim X'+b\undim X''=0
\]
\[
	\undim Y+\vare a\undim Y'+\vare'b\undim Y''=0
\]
Furthermore
\begin{enumerate}
	\item $\vare=1$ if and only if $X,X'$ both lie in $\cA$ or they both lie in $\cB$.
	\item $\Ext^1_\Lambda(X,X')\neq0$ if and only if $a>0$
	\item $\Ext^1_\Lambda(Y,Y')\neq0$ if and only if $\vare a>0$
\end{enumerate}
\end{lem}

\begin{proof}
The existence and uniqueness of $a,b\in\QQ$ follows from the fact that $\undim X',\undim X''$ are linearly independent vectors in $\QQ^n$ and $\undim X$ is in their span.

The equation for $\gamma$ gives us that $\undim Y$ is congruent to $\delta \undim X$ modulo $\undim T$ where $\delta=-1$ for $X\in\cA$ and $\delta=1$ for $X\in\cB$. This gives:
\[
	\delta \undim Y+\delta' a\undim Y'+\delta'' b\undim Y''=c\undim T
\]
for some $c\in\QQ$. However, $\undim Y',\undim Y'',\undim T$ are linearly independent since $(Y',Y'',T)$ form an exceptional sequence. And $\undim Y,\undim Y',\undim Y''$ are linearly dependent. So, $c=0$.

Statement (1) is now clear: $\delta,\delta'$ are both negative when $X,X'\in\cA$ and they are both positive if $X,X'\in\cB$. So, $\vare=\delta'/\delta=1$ in these cases and $\vare=-1$ in the other cases.

If $\Ext^1_\Lambda(X,X')\neq0$ then $a=1>0$. Otherwise, either $\Hom_\Lambda(X,X')\neq0$, in which case $a<0$ or both are zero in which case $a=0$. This proves (2) and (3) is similar.
\end{proof}
}

\begin{proof}[Proof of Proposition \ref{sigma preserves compatibility}]
{ 
Throughout this proof we will use the notation $Y[j]=\sigma(X[i])$, $Y'[j']=\sigma(X'[i'])$. By reversing the order if necessary we may assume that $i\le i'$ and $j\le j'$. There are four cases.

\vs2
\noi\underline{Case 1}: $i<i'$, $j<j'$.
\vs2

When $i<i'$, $X[i],X'[i']$ are compatible if and only if $(X,X')$ is an exceptional sequence. By the formula for $\sigma$, $(Y,Y')$ is equal to either $(X,X')$, $(X,\gamma(X'))$ or $(\gamma(X),\gamma(X'))$. Whichever it is, by Lemma \ref{three exceptional pairs}, $(Y,Y')$ is exceptional if and only if $(X,X')$ is exceptional. So, compatibility of $X[i],X'[i']$ is equivalent to compatibility of $Y[j],Y'[j']$.

\vs2
\noi\underline{Case 2}: $i<i'$, $j=j'$.
\vs2

In this case $i'=j+1$, $X[i]\in\cB_i$, $X'[i']\in \cA_{j+1}$, $Y'[j']\in \cA_{j}'$, there is a short exact sequence 
\[X'\into T^s\onto Y'\]
giving the correspondence $Y'=\alpha(X')$, and $Y$ is equal to $X$ or $\gamma(X)$. The case $Y=X\neq \gamma(X)$ occurs only when $i=k$, $\Hom_\Lambda(X,T)\neq0$ and $\gamma(X)$ is either the kernel or cokernel of the universal map $X\to T^s$.

Suppose that $X[i], X'[i']$ are compatible. Then $(X,X')$ is an exceptional pair. So, $(Y,Y')$ which is either $(X,Y')$ or $(\gamma(X),Y')$ is an exceptional pair by Lemma \ref{three exceptional pairs}. Then
\begin{enumerate}
\item $\Ext^1_\Lambda(\gamma(X),Y')=0$ since $Y'$ is a relatively injective object of $^\perp T$.
\item $\Ext^1_\Lambda(T,Y')=0$ since $\Hom_\Lambda(T,Y')\neq0$ and $(T,Y')$ is an exceptional pair.
\end{enumerate}
But $X,\gamma(X)$ and $T^s$ form an exact sequence in which $X$ is not the last item. Therefore, by right exactness of $\Ext_\Lambda^1(-,Y')$, we conclude that $\Hom_\Lambda(X,Y')=0$. Therefore, $Y,Y'$ are ext-orthogonal regardless of whether $Y=X$ or $Y=\gamma(X)$. So, $Y[j],Y'[j]$ are compatible.

Conversely, suppose that $Y[j],Y'[j]$ are compatible. Then we claim that $X[i],X'[i']$ are compatible, or equivalently $(X,X')$ is an exceptional pair. Since $X'$ is relatively projective in $T^\perp$, we have $\Ext^1_\Lambda(X',X)=0$. So, it suffices to show $\Hom_\Lambda(X',X)=0$. But $X\in T^\perp$. So,
\[
	\Hom_\Lambda(X',X)\cong \Ext^1_\Lambda(Y',X)\cong \Ext^1_\Lambda(Y',Y)=0
\]
where the second isomorphism follows from the fact that $Y'\in\,^\perp T$ and $X,Y,T^s$ form an exact sequence in which $X,Y$ are adjacent terms. (See Remark \ref{chart for bijection b}.)

\vs2
\noi\underline{Case 3}: $i=i'$, $j<j'$.
\vs2

Then $j=i-1$, we have an exact sequence $X\into T^s\onto Y$ giving the bijection $\alpha(X)=Y$ and $Y'=\beta(X')\in\,^\perp T$. If $Y[i-1], Y'[i]$ are compatible, then $(Y,Y')$ form an exceptional pair and thus $(X,X')$ also forms an exceptional pair. This implies that $X,X'$ are ext-orthogonal since $X$ is relatively projective in $T^\perp$. Thus $X[i],X'[i]$ are compatible.

Conversely suppose that $X[i],X'[i]$ are compatible. Since $Y'\in\,^\perp T$ we have, as in Case 2,
\[
	\Hom_\Lambda(Y',Y)\cong \Ext^1_\Lambda(Y',X)\cong \Ext^1_\Lambda(X',X)=0
\]
and $\Ext_\Lambda^1(Y',Y)=0$. Therefore $Y[i-1],Y'[i]$ are compatible.

\vs2
\noi\underline{Case 4}: $i=i'$, $j=j'$.
\vs2

In this case we need to prove that $X,X'$ are ext-orthogonal if and only if $Y,Y'$ are ext-orthogonal. The case $i<k$ being trivial, we assume that $i=i'\ge k$. 

Suppose that $i>k$. Then $Y=\gamma(X), Y'=\gamma(X')$. If $X,X'$ are ext-orthogonal they form an exceptional pair in some order, say $(X,X')$, and so does $(Y,Y')$. In the linear relation in Lemma \ref{linear relation} we are in the case when $\vare=1$ since $X,X$ are either both in $\cA$ or both in $\cB$. Therefore, $\Ext^1_\Lambda(X,X')=0$ if and only if $\Ext^1_\Lambda(Y,Y')=0$. So, $Y,Y'$ are ext-orthogonal. The converse works in the same way.

Finally, suppose that $i=k$. Then we also have $j=k$. The bijection $\beta_k:\cB_k\cong \cB_k'$ is given by $\beta$ and the argument in the previous paragraph applies except in Cases (c),(d) in the chart in the proof of Lemma \ref{lem: bijection bk}. Also, $\beta_k$ is the identity except in Case (b). So, the only case we need to check is when one of the objects, say $X[k]$ is in Case (b) and the other, $X'[k]$ is in Case (c) or (d). In other words, $Y=\beta(X)$ is given by the exact sequence
\[
	T^s\into Y\onto X
\]
and $Y'=X'\in T^\perp$ has the property that $\Ext^1_\Lambda(X',T)=0$ and $\Hom_\Lambda(X',T)\neq0$. The exact sequence tells us that
\[
	\Ext^1_\Lambda(X,X')\cong \Ext^1_\Lambda(Y,X')= \Ext^1_\Lambda(Y,Y').
\]
Also, $\Ext^1_\Lambda(Y',T)=0$ implies that
\[
	\Ext^1_\Lambda(Y',Y)\cong \Ext^1_\Lambda(Y',X)= \Ext^1_\Lambda(X',X).
\]
Therefore, $X[k],X'[k]$ are compatible if and only if $Y[k],Y'[k]$ are compatible. 
}
\end{proof}


\section{More properties of $m$-exceptional sequences}\label{sec4}

We show how $m$-exceptional sequences are related to the bijection between the set of (isomorphism classes of) $m$-cluster tilting objects in the $m$-cluster category of $mod\text-\Lambda$ and (isomorphism classes of) $m$-$\Hom_{\le0}$-configurations from \cite{9-BRT} which in the sequel we refer to as ``$m$-configurations''. We use this to derive the expected tropical duality formula (Corollary \ref{cor: duality between m-clusters and c-vectors}) relating the dimension vectors of these objects. Then we obtain a convenient reformulation of the mutation formula for $\tilde c$-vectors (Theorem \ref{thm: mutation given by tilde c vectors}) using a sign convention which matches $c$-vectors and dimension vectors of components of $m$-configurations. Finally, we give an example of this mutation formula. These formulas were first obtained in type $A_n$ in \cite{10-mtrees}

\subsection{Relation to Buan-Reiten-Thomas \cite{9-BRT}}\label{ss41}

First, the definitions. As always, $\Lambda$ is a finite dimensional hereditary algebra over a field $K$ with $n$ simple objects $S_i$.

\begin{defn}\label{def: mHom-configuration}\cite{9-BRT}
An \emph{$m$-configuration} in the bounded derived category $\cD^b(mod\text-\Lambda)$ is defined to be an object $M$ of this derived category with $n$ (nonisomorphic) components $M_i[k_i]$ satisfying the following conditions.
\begin{enumerate}
\item $0\le k_i\le m$
\item $\Hom_{\cD^b}(M_i[k_i],M_j[k_j-s])=0$ for all $i\neq j$ and $s\ge0$.
\item $M_1,\cdots,M_n$ form a complete exceptional sequence in some order.
\end{enumerate}
\end{defn}

In \cite{9-BRT} a bijection is constructed between the set of (isomorphism classes of) $m$-configurations and $m$-cluster tilting objects for any hereditary algebra $\Lambda$. This statement and proof will be reexamined in Theorem \ref{thm of BRT} below.

This bijection has been extended to a bijection between \emph{simple minded collections} and \emph{two term silting objects} in the category of bounded projective complexes \cite{11-IY}, \cite{12-KY}. However, the relation with exceptional sequences is special to the hereditary case.

We need the following definitions and sign conventions to make Definition \ref{def: mHom-configuration}, when expressed in terms of the dimension vectors of $M_j[k_j]$, agree (when $m=1$) with the characterization of $c$-vectors given in \cite{13-ST}.

\begin{defn}\label{def: m-slope vector}
The \emph{$(m)$-slope vector} of $M[k]$ is defined by $slope(M[k]):=\undim Mt^{m-k}\in\ZZ[t]^n$. We define the \emph{dimension vector} of $M[k]$ to be $\undim(M[k]):=(-1)^{m-k}\undim M$. This is the slope vector evaluated at $t=-1$. We define the \emph{$\tilde c$-vectors} of the $m$-cluster $T$ to be the slope vectors of the components $X_j[\ell_j]$ of the corresponding $m$-configuration:
\[
	\tilde c_j=slope(X_j[\ell_j])=\undim X_j\, t^{m-\ell_j}\in \ZZ[t]^n
\]
The \emph{$c$-vectors} of the $m$-cluster tilting object $T$ are $c_j=\tilde c_j|_{t=-1}=(-1)^{m-\ell_j}\undim X_j\in \ZZ^n$. We also use round brackets to indicate slope: 
\[
X(s):=X[m-s].\]
\end{defn}

The following is a reformulation of the main theorem of \cite{9-BRT} in terms of $m$-exceptional sequences:

\begin{thm}\cite{9-BRT}\label{thm of BRT}
Let $T=(T_1[k_1],T_2[k_2],\cdots,T_n[k_n])$ be an ordered $m$-cluster so that $(T_n,T_{n-1},\cdots,T_1)$ is a complete exceptional sequence. Then the $m$-exceptional sequence  corresponding to $T$ is equal to the set of components of the $m$-configuration corresponding to the $m$-cluster $T$.
\end{thm}

\begin{rem}\label{rem: after Thm of BRT}
Note that there may be more than one way to arrange the objects $T_j[k_j]$ so that $(T_n,\cdots,T_1)$ is a complete exceptional sequence. However, any two arrangements will differ by a sequence of transpositions of consecutive hom-ext-orthogonal objects. The statement of the theorem includes the statement that, if two consecutive $T_i$ commute in this way, the corresponding components $E_i$ of the $m$-configuration also commute (are hom-ext orthogonal). For example, one could arrange the terms $T_i[k_i]$ so that $k_1\ge k_2\ge\cdots\ge k_n$. However, assuming this would needlessly complicate the proof.
\end{rem}

We need to go through the proof (the same as in \cite{9-BRT}) since this proof leads us to the ``tropical duality formula'' Corollary \ref{cor: duality between m-clusters and c-vectors} below.

\begin{proof}
We follow the method in \cite{9-BRT} which also appears in \cite[Prop 3.3]{Br} and \cite{Bo}. Namely, we apply to the exceptional sequence $(T_n,T_{n-1},\cdots,T_1)$ the Garside braid move given by taking the objects on the left one at a time and moving them over the objects on the right which haven't moved yet. The result is $(E_1,\cdots,E_n)$ where, by construction, $E_n=T_n$. The slopes are determined by a simple formula. The statement of the theorem is that this is a special case of the bijection between $m$-cluster tilting objects and $m$-exceptional sequences.

We recall that the correspondence $\theta_n:T\mapsto E$ is given by induction on $n$ by
\[
	\theta_n\left(T_1[k_1],\cdots,T_n[k_n]\right)
	= \left(\theta_{n-1}(X_1[\ell_1], \cdots, X_{n-1}[\ell_{n-1}]), T_n[k_n]\right)
\] 
where $X_j[\ell_j]=\sigma^{-1}_{T_n[k_n]}(T_j[k_j])$. Recall from Proposition \ref{prop summarizing bijection sigma-T[k]} that $X_j[\ell_j]$ is uniquely determined by the property that $(T_n,T_j)\sim (X_j,T_n)$ and the slope of $X_j$ is equal to or one less than the slope of $T_j$. In other words, $\theta_n$ is the same as the bijection given by action of the Garside braid on the exceptional sequence $(T_n,\cdots,T_1)$. We note that, by the properties of the braid group action on exceptional sequences, $(X_{n-1},\cdots,X_1)$ is an exceptional sequence and, by the properties of the bijection $\sigma_{T_n[k_n]}$, $X_i$ are compatible (mutually ext-orthogonal). Furthermore, $X_i,X_{i+1}$ are hom-ext-orthogonal if and only if $T_i,T_{i+1}$ are hom-ext-orthogonal. Therefore, $X_i[\ell_i]$ satisfy the induction hypotheses. So, the process can be iterated and the proof is complete.
\end{proof}

\subsection{Tropical duality}\label{ss42}

From the proof of Theorem \ref{thm of BRT} we obtain a version of the tropical duality formula for $m$-clusters. (Corollary \ref{cor: duality between m-clusters and c-vectors} below.)

\begin{lem}\label{lem: slope of th(T) is = or 1 less than slope of T}
In the correspondence $E=\theta(T)$, the slope of each component $E_j$ is either equal to or one less than the slope of $T_j$.
\end{lem}

\begin{proof} Let $X_j$ be as in the proof of Theorem \ref{thm of BRT}. By Remark \ref{X is rel projective}, if the slope of $X_j$ differs from the slope of $T_j$ then $X_j$ is relatively projective in the smaller category $T_n^\perp$ and $slope\,X_j=slope\,T_j-1$. Also, $T_j$ cannot be projective since there is a nonsplit short exact sequence $X_j\into T_n^s\onto T_j$. Thus, in the next step, when $X_j$ is replaced by $X_j'\in (T_n\oplus T_{n-1})^\perp$, $X_j'$ must have the same slope as $X_j$ if $X_j$ is projective. By Lemma \ref{extension to k=m}, once an object becomes relatively projective, it will remain relatively projective through the later steps. Thus, in the sequence $T_j\mapsto X_j\mapsto X_j'\mapsto\cdots\mapsto E_j$, the slope cannot change more than once.
\end{proof}

\begin{cor}\label{cor: duality between m-clusters and c-vectors}
The dimension vectors of components of an $m$-cluster tilting object $T=\bigoplus T_i[k_i]$ and of the dual $m$-configuration $X=\bigoplus X_j[\ell_j]$ are related by the formula:
\[
	V^tEC=D
\]
where $V$ has columns $\undim T_i[k_i]=(-1)^{m-k_i}\undim T_i$, $C$ has columns $c_j=(-1)^{m-\ell_j}\undim X_j$, $D$ is the diagonal matrix with diagonal entries $f_k=\dim_K\End(T_k)$ and $E$ is the Euler matrix with entries $e_{ij}=\dim_K\Hom(S_i,S_j)-\dim_K \Ext^1(S_i,S_j)$.
Furthermore, the slope of each $T_j$ is either equal to or one more than the slope of $\tilde c_j$, whichever gives the correct sign for the pairing $\left<\undim T_j[k_j], c_j\right>=+f_j$.
\end{cor}

\begin{rem}\label{rem: equivalent to VtED=D}
The formula $V^tEC=D$ is equivalent to the \emph{tropical duality} formula $G^tDC=D$ from \cite{14-NZ} where $G$ is the matrix of $g$-vectors (times $(-1)^m$ to match the sign shift of $C$. In \cite{15-HVfans} we shift the sign of $C$ but not of $G$ so that $G^tDC=(-1)^mD$.). The equations are equivalent since $V^t=G^tDE^{-1}$. This follows from the well-known fact that the rows of $DE^{-1}$ are the dimension vectors of the indecomposable projective $\Lambda$-modules and the $k$th $g$-vector gives the projective presentation of the $k$th component $T_k$ of the $m$-cluster tilting object $T$.
\end{rem}

\begin{proof}
This follows from the proof of Theorem \ref{thm of BRT} where we recall that $\theta(T)=(E)=(E_1,\cdots,E_n)$ and $X_k$ are given by the braid move $(T_n,T_k)\sim (X_k,T_n)$.

Claim 1: $\brk{\undim T_i,\undim E_j}=0$ if $i< j$ where $\brk{x,y}:=x^tEy$.

Proof: Since $\undim E_j$ is a linear combination of $\undim X_k$ for $k\ge j$ which are linear combinations of $\undim T_k$ by induction, $\brk{\undim T_i,\undim E_j}$ is a linear combination of $\brk{\undim T_i,\undim T_k}$ for $k\ge j>i$. But these are all zero since $(T_n,\cdots,T_1)$ is an exceptional sequence.

Claim 2: $\brk{\undim T_i,\undim E_j}=0$ if $i> j$.

Proof: By induction on $n-i$. First, let $i=n$. Then, for $j\le k<n$, $X_k\in T_n^\perp$ for $X_k[\ell_k]=\sigma_{T_n[k_n]}^{-1}(T_j[k_j])$. So, $\brk{\undim T_i,\undim X_k}=0$. But $\undim E_j$ is a linear combination of $\undim X_k$ for $j\le k<n$. So, $\brk{\undim T_i,\undim E_j}=0$. 

Next, suppose $j<i<n$. Then, by induction we have $\brk{\undim T_k,\undim E_j}=0$ for all $i<k\le n$ and $\brk{\undim X_i,\undim E_j}=0$ since $(n-1)-i<n-i$. But the braid relation $(T_n,T_i)\sim (X_i,T_n)$ implies that, either (a) $X_i=T_i$, or (b) there is a short exact sequence 
$
X_i\into T_n^s\onto T_i$ or $T_i\into X_i\onto T_n^s
$. In any case we have
\begin{equation}\label{eq: Xi in terms of Ti}
	\undim T_i=\delta \undim X_i+s\undim T_n
\end{equation}
for $\delta=\pm1$ and some integer $s$. So,
\[
	\brk{\undim T_i,\undim E_j}=\delta \brk{\undim X_i,\undim E_j}+s\brk{\undim T_n,\undim E_j}=0.
\]

Claim 3: For $i=j$, $\brk{\undim T_i,\undim E_i}=\pm f_i$ where $f_i=\dim_K \End(T_i)$.

Proof. By iteration the formula \eqref{eq: Xi in terms of Ti} we see that $\undim T_i$ is equal to $\delta \undim E_i$ plus a linear combination of $\undim T_k$ for $i<k\le n$. Since $\brk{\undim T_i,\undim T_k}=0$ for all such $k$ we get
\[
	f_i=\brk{\undim T_i,\undim T_i}=\delta\brk{\undim T_i,\undim E_i}
\]
as claimed.

Finally, we note that the sign $\delta=\pm1$ in \eqref{eq: Xi in terms of Ti} is equal to $-1$ only when $T_i,X_i$ are related by a short exact sequence of the form $X_i\into T_n^s\onto T_i$. In this case the slope of $X_i$ is one less than the slope of $T_i$ and this can happen only once for each $i$. Therefore, either $\brk{\undim T_i,\undim E_i}=f_i$ and $T_i,E_i$ have the same slope or $\brk{\undim T_i,\undim E_i}=-f_i$ and the slope of $E_i$ is one less than the slope of $T_i$. The proof of the duality formula is complete.
\end{proof}

\subsection{Mutation formula for $\tilde c$-vectors}\label{ss43} For each $m$-configuration $X$ we construct a sequence of overlapping wide subcategories $H_s(X)$ of $mod\text-\Lambda$ which we call ``horizontal'' or ``vertical'' subcategories depending on the parity of $s$. Then we show that each mutation of $X$ can be construed to take place in the cluster category of one of these subcategories.

In the $m$-cluster category of $mod\text-\Lambda$, each component $T_k$ of each $m$-cluster tilting object $T$ can be replaced with one of $m$ other objects which are given by $\mu_k^+(T_k), (\mu_k^+)^2(T_k),\cdots, (\mu_k^+)^s(T_k)$ and $\mu_k^-(T_k), (\mu_k^-)^2(T_k)\cdots, (\mu_k^+)^{m-s}(T_k)$ where the formulas for the ``positive'' and ``negative'' mutation operators $\mu_k^+$, $\mu_k^-$ are defined using distinguished triangles \cite{11-IY}. We use the sign convention given by slope. We recall that there are distinguished triangles
\[
	T_k\to B\to \mu_k^-(T_k)\to T_k[1],\quad T_k[-1]\to \mu_k^+(T_k)\to B'\to T_k
\]
where $B,B'$ are left/right $add\, T/T_k$-approximations of $T_k$. The slope of $\mu_k^+(T_k)$ is greater than or equal to the slope of $T_k$ and analogously for $\mu_k^-(T_k)$.

The mutation formula is better behaved for the corresponding $\tilde c$-vectors since $\mu_k^+$ always increases the slope of $\tilde c_k$ by one and $\mu_k^-$ always decreases the slope of $\tilde c_k$ by one. In fact, positive and negative mutation of $\tilde c$ is given by the following formula which needs a definition first. 

Let $\vare\beta\in\ZZ^n$ be a real Schur root for $\Lambda$ where $\vare=\pm1$ is its sign and $\beta$ is a positive real Schur root. Then we define the vector $\psi_s(\vare\beta)\in \ZZ[t]^n$ to be $t^s\beta$ if $\vare=+1$ and $t^{s+1}\beta$ if $\vare=-1$.

\begin{rem}\label{Hs(X) is intersection of perp cats}
Given $X$ an $m$-configuration and consecutive integers $s<s+1$, let $\cA_s(X)$ denote wide subcategory of $mod\text-\Lambda$ spanned by modules $M_i$, $i=1,\cdots,h_s$, so that either $M_i(s)$ or $M_i(s+1)$ is a component of $X$ (Recall that $M_i(s)=M_i[m-s]$). Then $\cA_{s}(X)$ is a rank $h_s$ wide subcategory of $mod\text-\Lambda$ so that $\cA_s(X)\subseteq \cA_t(X)^\perp$ for all $t\ge s+2$. In fact $\cA_s(X)$ is the intersection of the perpendicular categories 
\begin{equation}\label{eq: Hs(X) as intersection of perpendicular categories}
\cA_s(X)=
\bigcap_{t\ge s+2} \cA_t(X)^\perp \cap 
\bigcap_{r\le s-2} \,^\perp \cA_r(X) .
\end{equation}
Note that, in general, $\cA_s\cap \cA_t$ is nonzero only when $|t-s|\ge2$.
\end{rem}

For $i=1,\cdots,h_s$, let $c_i(s)=\undim M_i$ or $-\undim M_i$ depending on whether $M_i(s)$ or $M_i(s+1)$ is a component of $X$, respectively. Then Definition \ref{def: mHom-configuration} implies that the vectors $c_i(s)$ satisfy the condition of \cite{13-ST} and therefore form a set of $c$-vectors for some cluster tilting object for $\cA_s(X)$. Let $Z(s)\cong \ZZ^{h_s}$ be the additive subgroup of $\ZZ^n$ freely generated by these $h_s$ elements. Then, for any positive real Schur root $\beta$ of $\Lambda$, $\beta\in Z(s)$ if and only if the corresponding exceptional $\Lambda$-module $M_\beta$ lies in the wide subcategory $\cA_s(X)$. The reason is the $\cA_s(X)$ is given by the linear condition \eqref{eq: Hs(X) as intersection of perpendicular categories} whose integer solution set is $Z(s)$. We define $\psi_s(\beta), \psi_s(-\beta)$ to be the objects in the bounded derived category of $\Lambda$ equal to $M_\beta$ at level $[m-s]$ (with slope $s$) or level $[m-s-1]$ with slope $s+1$, respectively:
\[
	\psi_s(\beta)=M_\beta(s),\ \psi_s(-\beta)=M_\beta(s+1)
\]
Recall that $X=\bigoplus M_i[k_i]$ is a fixed $m$-configuration.

\begin{lem}\label{lem: get another m-configuration}
Let $v_i\in Z(s)\subset\ZZ^n$, $i=1,\cdots,h_s$ satisfying the following.
\begin{enumerate}
	\item Each $v_i$ is a positive or negative real Schur root of $\Lambda$. Let $N_i$ or $N_i[1]$, resp., be the corresponding object of the derived category of $\cA_s(X)$.
	\item After possibly rearranging the terms, $(N_1,\cdots,N_{h_s})$ form an exceptional sequence with  negative terms (terms corresponding to negative $v_i$) coming after all positive terms.
\end{enumerate} 
Then, the direct sum of the $h_s$ objects $\psi_s(v_i)$ together with all components of $X$ of slope not equal to $s$ or $s+1$ is another $m$-configuration for $\Lambda$.
\end{lem}

\begin{proof}
This follow directly from the definition of an $m$-configuration.
\end{proof}

Recall that the $\tilde c$-vectors of an $m$-cluster tilting object of $\Lambda$ are the refined dimension vectors $\tilde c_i=\undim M_i\,t^{s_i}\in \ZZ[t]^n$ for the components $M_i(s_i)=M_i[m-s_i]$ of the $m$-configuration $X$.

\begin{thm}\label{thm: mutation given by tilde c vectors}
Suppose that $\tilde c_k$ has slope $s$ (resp. $s+1$). Then the positive (resp. negative) mutation $X'=\mu_k^+(X)$ (resp. $X'=\mu_k^-(X)$) of $X$ is uniquely determined by its $\tilde c$-vectors $\tilde c_j'$ which are related to the $\tilde c$-vectors $\tilde c_j$ of $X$ as follows.
\begin{enumerate}
\item $\tilde c_k'=\tilde c_kt$ (resp. $\tilde c_k'=\tilde c_kt^{-1}$).
\item $\tilde c_j'=\tilde c_j$ if the slope of $\tilde c_j$ is not equal to $s$ or $s+1$.
\item $\tilde c_j'=\tilde c_j$ if $b_{kj}\le0$ (resp. $b_{kj}\ge0$).
\item When $b_{kj}>0$ (resp. $b_{kj}<0$) and $\tilde c_j$ has slope either $s$ or $s+1$ then:
\begin{enumerate}
	\item $\tilde c_j'$ has slope $s$ or $s+1$.
	\item $c_j'=c_j+|b_{kj}|c_k$ in $\ZZ^n$.
\end{enumerate}
\end{enumerate}
\end{thm}

\begin{proof}[Proof that this works:]

It satisfies the definition of \cite{9-BRT} which is the following.
\begin{enumerate}
\item components of the same slope are hom-orthogonal
\item components can be arranged in an exceptional sequence
\item components of smaller slope come before components of larger slope in the exceptional sequence.
\end{enumerate}
We show that these properties hold for $X'$, the mutated $X=\bigoplus X_j$.

Suppose that $\delta=+$ and $X_k$ has slope $s$. Then only those components of $X$ of slopes $s,s+1$ are changed and the number of them, say $h$, is unchanged. Let $X_1,\cdots,X_h$ denote these components of $X$ and let $X_j=M_j(\ell_j)=M_j[m-\ell_j]$ where $M_j$ is a $\Lambda$-module and $\ell_j\in \{s,s+1\}$. Let $\cA_s$ be the smallest wide subcategory of $mod\text-\Lambda$ containing the $h$ modules $M_j$. Then $\cA_s\cong mod\text-\Lambda_s$ for some hereditary algebra $\Lambda_s$ of rank $h$ and the vectors $v_j=(-1)^{\ell_j-s}\undim M_j$ form the set of $c$-vectors of a cluster tilting object for $\Lambda_s$ by \cite{13-ST} (since we use the sign convention designed to make this step work).

The mutation formula ($\ast$) is equivalent, by the correspondence $\tilde c_j\leftrightarrow v_j$, to the standard mutation formula for $c$-vectors for $\Lambda_s$. Furthermore, all components of $X$ of slope less than $s$, resp. larger than $s+1$, are right perpendicular, resp. left perpendicular, to all objects in $\cA_s=mod\text-\Lambda_s$. By Lemma \ref{lem: get another m-configuration}, the new $\tilde c_j$ vectors satisfy the required conditions.

Let $X'$ be the mutated $m$-configuration. Then the dual $T'$ of $X'$ clearly has the property that it differs from the dual $T$ of $X$ only in its $k$-th component $T_k$:

(a) For $i\neq j,k$, 
\[
	\brk{\undim T_i,c_j'}=\brk{\undim T_i,c_j}+|b_{kj}|\brk{\undim T_i,c_k}=0
\]
\[
	\brk{\undim T_i,c_k'}= -\brk{\undim T_i,c_k}=0 
\]

(b) For $i=j$,
\[
	\brk{\undim T_j,c_j'}=\brk{\undim T_j,c_j}+|b_{kj}|\brk{\undim T_j,c_k}=f_j
\]
So, $T_i$ satisfies the tropical equations (Corollary \ref{cor: duality between m-clusters and c-vectors}) characterizing $T_i'$ and, thus, $T_i'=T_i$ for all $i\neq k$. The equations which characterize $T_k'$ are:
\[
	\brk{\undim T_k',c_j'}=\brk{\undim T_k',c_j}+|b_{kj}|\brk{\undim T_k',c_k}=0
\]
\[
	\brk{\undim T_k',c_k'}=-\brk{\undim T_k',c_k}=f_k
\]
If we write $\undim T_k'=\sum a_j \undim T_j$ then the second equation gives $a_k=-1$ and the first equation gives:
\[
	a_jf_j=\begin{cases} |b_{kj}|f_k=\delta(\brk{c_j,c_k}-\brk{c_k,c_j}) & \text{if the sign of $b_{kj}$ is $\delta$}\\
    0& \text{otherwise}
    \end{cases}
\]
So,
\[
	a_j=\begin{cases} |b_{jk}| & \text{if the sign of $b_{kj}$ is $\delta$}\\
    0& \text{otherwise}
    \end{cases}
\]
and, by Corollary \ref{cor: duality between m-clusters and c-vectors}, $T_k'$ is uniquely determined by its dimension vector.
\end{proof}

In the next paper \cite{15-HVfans} the wide subcategories $\cA_s$ are used to give a visualization of $m$-clusters analogous to the semi-invariant pictures \cite{16-ITW16} (also known as ``scattering diagrams'') used in \cite{17-BHIT} to derive properties of maximal green sequences. The old Section 2 has been extracted into a separate paper \cite{AnProb}.

\section*{Acknowledgements}

Some of the results of this paper and its sequel \cite{15-HVfans} were announced in a lecture at the Workshop on Cluster Algebras and Related Topics at the Chern Institute at Nankai University in July, 2017. The author would like to thank Fang Li, Bin Zhu and the other organizers of this wonderful meeting where the results of this paper took final form. The author also thanks Steve Hermes, Keith Merrill, Theo Douvropoulos and Shujian Chen for many discussion about this work and its sequel which is being prepared. The author thanks the Simons Foundation for its support: Grant \#686616.


\begin{thebibliography}{aa}


\bibitem{Bo} Alexey I. Bondal, \emph{Representation of associative algebras and coherent sheaves}, Math. USSR Izvestiya, Vol 34(1990), No 1, 23--42. 

\bibitem{Br} Tom Bridgeland, \emph{$t$-structures on locally Calabi-Yau varieties}, J. Algebra 289 (2005), 453--483. 

\bibitem{3-BMRRT}
Aslak~Bakke Buan, Robert~J. Marsh, Marcus Reineke, Idun Reiten, and Gordana Todorov, \emph{Tilting theory and cluster combinatorics}, Adv. Math. \textbf{204} (2006), no.~2, 572--618. 


\bibitem{9-BRT}
Aslak~Bakke Buan, Idun Reiten and Hugh Thomas, \emph{From $m$-clusters to $m$-noncrossing partitions via exceptional sequences}, Mathematische Zeitschrift 271 (2012), no. 3-4, 1117--1139. 

\bibitem{17-BHIT}
Thomas Br\"ustle, Stephen Hermes, Kiyoshi Igusa and Gordana Todorov, \emph{Semi-invariant pictures and two conjectures on maximal green sequences}, J Algebra {\bf 473}, March 2017, 80--109. 


\bibitem{5-Ch} F. Chapoton, \emph{Enumerative properties of generalized associahedra}, S\'eminaire Lotharingien de Combinatoire 51 (2004): B51b. 

\bibitem{6-CB}
William Crawley-Boevey. \emph{Exceptional sequences of representations of quivers}. In Representations of
algebras (Ottawa, ON, 1992), volume 14 of CMS Conf. Proc., pages 117--124. Amer. Math. Soc., Providence, RI, 1993. 

\bibitem{D} Theo Douvropoulos, \emph{On enumerating factorizations in reflection groups.} Algebr. Comb. 6.2 (2023), pp. 359--385. issn: 2589--5486. 

\bibitem{4-FR}
Sergey Fomin and Nathan Reading, \emph{Generalized cluster complexes and Coxeter combinatorics}, International Mathematics Research Notices 44 (2005), 2709--2757. 

\bibitem{GLTW} Pavel Galashin, Thomas Lam, Minh-T\^am Trinh and Nathan Williams. \emph{Rational Noncrossing Coxeter-Catalan Combinatorics}. arXiv preprint arXiv:2208.00121 (2022). 

\bibitem{CD} Guillaume Chapuy and Theo Douvropoulos. \emph{Counting chains in the noncrossing partition lattice via the W-Laplacian.} Journal of Algebra 602 (2022): 381--404. 

\bibitem{CI} Shujian Chen and Kiyoshi Igusa, \emph{All terms in a complete exceptional sequence are relatively projective or relatively injective}. arXiv preprint arXiv:2312.05997 (2023). 

\bibitem{1-IT13} Kiyoshi Igusa and Gordana Todorov, \emph{Signed exceptional sequences and the cluster morphism category}, arXiv:1706.02041. 

\bibitem{16-ITW16} 
Kiyoshi Igusa, Gordana Todorov, and Jerzy Weyman, \emph{Picture groups of finite type and cohomology in type $A_n$}, arXiv:1609.02636. 

\bibitem{10-mtrees} Kiyoshi Igusa, \emph{$m$-noncrossing trees}, Journal of Algebra and Its Applications 17.10 (2018): 1850187. 

\bibitem{15-HVfans} \bysame, \emph{Horizontal and vertical mutation fans}, Sci China Math, 2019, 62: 1267--1288. 

\bibitem{AnProb} \bysame, \emph{Probability distribution for exceptional sequences of type $A_n$}, arXiv:2112.04996. 

\bibitem{11-IY} Osamu Iyama and Yuji Yoshino, \textit{Mutation in triangulated categories and rigid Cohen-Macaulay modules}, Inventiones mathematicae 172, no. 1 (2008): 117--168. 


\bibitem{12-KY} S. Koenig and D. Yang, \emph{Silting objects, simple-minded collections, t-structures and co-t-structures for finite-dimensional algebras}, Documenta Mathematica, 19 (2014): 403--438. 


\bibitem{M} Jean Michel, \emph{Deligne-Lusztig theoretic derivation for Weyl groups of the number of reflection factorizations of a Coxeter element.} Proc of Amer Math Soc 144.3 (2016): 937--941. 

\bibitem{14-NZ} Tomoki Nakanishi, Andrei Zelevinsky, \emph{On tropical dualities in cluster algebras}, Algebraic Groups and Quantum Groups, Contemp. Math 565 (2012): 217--226. 

\bibitem{2-Ringel13} A. A. Obaid, S. K. Nauman, W. S. Al Shammakh, W. M. Fakieh and C. M. Ringel: \emph{The
number of complete exceptional sequences for a Dynkin algebra}, Colloq. Math. 133 (2013), 197--210. 


\bibitem{7-Ringel}
Claus Michael Ringel, \emph{The braid group action on the set of exceptional sequences of a hereditary Artin algebra} In Abelian group theory and related topics (Oberwolfach, 1993), volume 171 of Contemp. Math., pages 339--352. Amer. Math. Soc., Providence, RI, 1994. 


\bibitem{13-ST}
David Speyer and Hugh Thomas, \emph{Acyclic cluster algebras revisited}, ``Algebras, quivers and representations, Proceedings of the Abel Symposium 2011 (2013), 275--298. 


\bibitem{8-mT}
 Hugh Thomas, \emph{Defining an $m$-cluster category}, J. Algebra 318 (2007), no. 1, 37--46. 

\end{thebibliography}
\end{document}